\documentclass[11pt]{article}

\usepackage{amsfonts,amsmath,amsthm,latexsym,color,epsfig,mathrsfs,enumerate}
\newtheorem{thm}{Theorem}%[section]

\newtheorem{lemma}{Lemma}[section]
\newtheorem{thmtool}{Theorem}[section]

\newtheorem{lem}[thmtool]{Lemma}
\newtheorem{prop}[thmtool]{Proposition}

\newtheorem*{conjecture*}{Conjecture}
\newtheorem*{question*}{Question}

\newcommand{\A}{\mathcal{A}}
\newcommand{\B}{\mathcal{B}}
\newcommand{\pat}[2]{\operatorname{pat}(#1,#2)}
\newcommand{\inte}[1]{\operatorname{IP}(#1)}
\newcommand{\alt}[1]{\operatorname{AP}(#1)}
\newcommand\numberthis{\addtocounter{equation}{1}\tag{\theequation}}

\title{Set families with a forbidden pattern}
\author{Ilan Karpas \thanks{The Einstein Institute of Mathematics, 
The Hebrew University of Jerusalem, Jerusalem, Israel. Email: ilan.karpas@mail.huji.ac.il.} \and Eoin Long \thanks{School of Mathematical Sciences, Tel Aviv University, Tel Aviv, Israel. E-mail: eoinlong@post.tau.ac.il.}}
\date{}

\begin{document}

\maketitle

\begin{abstract}
	A \emph{balanced pattern} of order $2d$ is an element 
	$P \in \{+,-\}^{2d}$, where both signs appear $d$ times. Two sets 
	$A,B \subset [n]$ form a $P$-pattern, which we denote by 
	$\pat A B = P$, if $A\triangle B = \{j_1,\ldots ,j_{2d}\}$ 
	with $1\leq j_1<\cdots < j_{2d}\leq n$ and $\{i\in [2d]: P_i = + \} 
	= \{i\in [2d]: j_i \in A \setminus  B\}$. We say ${\cal A} 
	\subset {\cal P}[n]$ is $P$-free if $\pat A B \neq P$ for all 
	$A,B \in {\cal A}$. We consider the following extremal 
	question: how large can a family ${\cal A} \subset {\cal P}[n]$ 
	be if $\A $ is $P$-free?
	
	We prove a number of results on the sizes of such families. 
	In particular, we show that for some fixed $c>0$, if $P$ 
	is a $d$-balanced pattern with 
	$d < c \log \log n$ then $|\A | = o(2^n)$. We then give stronger bounds in the cases when 
	(i) $P$ consists of $d$ $+$ signs, followed by $d$ $-$ signs 
	and (ii) $P$ consists of alternating signs. In both cases, if 
	$d = o(\sqrt n)$ then $|\A | = o(2^n)$. In the case of (i), 
	this is tight.		 
\end{abstract}

\section{Introduction}

A central goal in extremal set theory is to understand how large a set family 
 can be subject to some restriction on the intersections of its elements. Given ${\cal L} \subset {\mathbb N} \cup \{0\}$, we say that a set family $\A $ is ${\cal L}$-intersecting if $|A \cap B| \in {\cal L}$ for all $A,B \in \A $. Taking ${\cal L}_t =\{s \in {\mathbb N}: s\geq t\}$, a fundamental theorem of Erd\H{o}s, Ko and Rado \cite{Erdos-Ko-Rado} shows that ${\cal L}_t$-intersecting families $\A \subset \binom {[n]}{k}$ satisfy $|\A | \leq  \binom {n-t}{k-t}$, provided $n\geq n_0(k,t)$.
Another important theorem due to Frankl and F\"uredi \cite{Frankl-Furedi} shows that if ${\cal L}_{\ell , \ell '} := \{s < \ell \mbox { or } s\geq k-\ell '\} $, then any ${\cal L}_{\ell ,\ell '}$-intersecting family ${\cal A} \subset \binom {[n]}{k}$ satisfies $|\A | \leq c n^{\max (\ell , \ell ')}$, for some constant $c$ depending on $k, \ell$ and $\ell '$. See \cite{Babai-Frankl}, \cite{Bollobas}, \cite{Frankl}, \cite{Furedi} for an overview of this extensive topic.
 
Here we are concerned with understanding the effect of restricting the \emph{pattern} formed between elements of a set family. A \emph{difference pattern} or \emph{pattern} of order $t$ is an element $P \in \{+,-\}^t$. Given such a pattern $P$, let $S_+(P) = \{i\in [t]: P_i = +\} \subset [t]$ and $s_+(P) = |S_+(P)|$. Define $S_-(P)$ and $s_-(P)$ analogously. Two sets $A,B \subset [n]$ form a \emph{difference pattern} $P$ if: 
	\begin{enumerate}[(i)]
		\item $A \triangle B = \{j_1,\ldots ,j_t\}$ with $j_1<\cdots < j_t$, and
		\item $\{i\in [t]: P_i = + \} = \{i\in [t]: j_i \in A \setminus  B\}$.
	\end{enumerate} 
We denote this by writing $\pat A B = P$. A family of subsets $\A \subset {\cal P}[n]$ is \emph{$P$-free} if $\pat A B \neq P$ for all $A,B \in {\cal A}$. In this paper we consider the following natural question: given a pattern $P$, how large can a family ${\cal A} \subset {\cal P}[n]$ be if it is $P$-free? 

First note the following simple observation. If $s_+(P) \neq s_-(P)$ then large $P$-free families exist. Indeed, if $|s_+(P) - s_-(P)| = m >0$ then the following families are ${P}$-free: $${\cal B}_1 = \{A \subset [n]: |A| \in [0,m-1] \mod 2m\}; \quad {\cal B}_2 = \{A \subset [n]: |A| \in [m,2m-1] \mod 2m \}.$$ Clearly either $|{\cal B}_1| \geq 2^{n-1}$ or $|{\cal B}_2| \geq 2^{n-1}$. We will therefore focus on the case when $s_+(P) = s_-(P) = d$. We say that such patterns are $d$-\emph{balanced}. For a balanced pattern $P$ it is only possible that $\pat A B = P$ if $|A| = |B|$. Thus, our question on balanced patterns essentially reduces to a question for uniform families. Given $0\leq k \leq n$, define 
$$f(n,k,P) := \max \Big \{ |\A |: P\mbox{-free families } \A \subset \binom {[n]}{k} \Big \}.$$ Let $f(n,k,d) = \max \{ f(n,k,P): P \mbox{ is }d\mbox{-balanced}\}$. We will also write $\delta (n,k,P)$ and $\delta (n,k,d)$ for the corresponding extremal densities, i.e. $\delta (n,k,P) := f(n,k,P) /\binom {n}{k}$, 
and $\delta (n,k,d) := f(n,k,d) /\binom {n}{k}$.  Note also that if ${\cal A} \subset \binom {[n]}{k}$ is $P$-free then the family ${\cal A}^c = \{[n]\setminus A: A\in {\cal A}\} 
\subset \binom {[n]}{n-k}$ is also $P$-free. Therefore $f(n,k,P) = f(n,n-k,P)$ and it suffices to bound $f(n,k,P)$ for $k\leq n/2$.

Our first aim is to prove a density result for $d$-balanced patterns of small order. In this context, first note that for any fixed $k \geq 2d-1$ and taking $\ell = k -d$ and $\ell ' = d-1$, the Frankl-F\"uredi theorem shows that if $\A \subset \binom {[n]}{k}$ with $|\A | \gg n^{k-d}$ then there are $A,B \in \A $ with $|A \triangle B| = 2d$, i.e. $A$ and $B$ form a $P$-pattern for \emph{some} $d$-balanced pattern $P$. It therefore natural to ask whether we already get density results for all fixed patterns for $k\geq k_0(d)$ large enough? This is easily seen to be false however. For any $k\in \mathbb N$, consider the family $\A _0\subset \binom {[n]}{k}$ given by 
$$\A _0= \Big \{A \subset [n]: \Big |A \cap \Big [\frac{(i-1)n}{k},\frac {in}{k} \Big ) \Big | = 1 \mbox{ for all }i\in [k] \Big \}.$$ 
Then $|\A _0| \geq c_kn^k$ for some absolute constant $c_k >0$, but it is easily seen that $\A _0$ does not contain the pattern $++--$. Therefore, there does not exist a density theorem for $d$-balanced patterns in subsets of $\binom {[n]}{k}$ with fixed $k$, as in the Frankl-F\"uredi theorem. 

Our first result shows that such a density theorem does hold for $k$ growing with $n$.

\begin{thm} 
	\label{thm: general pattern theorem}
	Given $d,k,n \in {\mathbb N}$ with $2k \leq n$ and 
	taking $a_d= (8d)^{5d}$ and $c_d = 6d8^{-d}$ we have 
	$$\delta (n,k,d) \leq a_{d} k^{-c_d}.$$
\end{thm}
By our discussion above for fixed $k$ we see that Theorem \ref{thm: general pattern theorem} is in a sense a `high-dimensional' result. Also note that Theorem \ref{thm: general pattern theorem} shows there is a constant $c>0$ with the property that if $P$ is a $d$-balanced pattern with $d \leq c \log \log n$ and $\A \subset {\cal P}[n]$ which is $P$-free, then $|{\cal A}| = o(2^n)$.

Let $\inte d$ denote the $d$-balanced pattern consisting of $d$ plus signs, followed by $d$ minus signs. We refer to these as \emph{interval patterns}. Given the obstruction of $\inte 2$  above, it is natural to ask for bounds on $f(n,k,\inte d)$. 
\begin{thm}
	\label{thm: interval pattern theorem}
	Given $d,k,n \in {\mathbb N}$ with $2k \leq n$ we have 
	$$\delta (n,k,\inte d) = O (d^2k^{-1} ).$$
\end{thm}
In particular, families $\A \subset {\cal P}[n]$ which are $\inte d$-free 
for $d \ll \sqrt n$ satisfy $|\A | = o(2^n)$. Furthermore, this turns out to be tight -- if $d \geq c\sqrt n$ then there are $\inte d$-free families $\A \subset {\cal P}[n]$ with $|\A | = \Omega _c(2^n)$.

Lastly, we consider the $d$-balanced pattern $\alt d$ consisting of alternating plus and minus signs, e.g. $\alt 2 = +-+-$. We refer to these as \emph{alternating patterns}. Our next result proves a density result for such patterns. 
\begin{thm}
	\label{thm: alternating pattern theorem}
	Given $d,k,n \in {\mathbb N}$ with $2k \leq n$ we have 
	$$\delta (n,k,\alt d) = 
	O \Big (\log ^{-1}\Big (\frac {k}{d^2} \Big ) \Big ).$$
\end{thm}
Thus again, all families $\A \subset {\cal P}[n]$ which are $\alt d$-free for $d \ll \sqrt n$ satisfy $|\A | = o(2^n)$. 
Unlike in the case of the interval patterns, we do not know if this is tight. 

Before closing the introduction, we mention some further results related to this topic. A family $\A \subset {\cal P}[n]$ is said to be a \emph{tilted Sperner family} if for all distinct $A,B \in {\cal A}$ we have $|B \setminus A| \neq 2|A \setminus B|$. 
Equivalently, ${\cal A}$ is $P$-free for all patterns $P$ with 
$|S_-(P)| = 2|S_+(P)|$. Kalai raised the question of how large a tilted Sperner family ${\cal A} \subset {\cal P}[n]$ can be. In \cite{Leader-Long}, Leader and the second author proved that such families satisfy $|\A | \leq (1+ o(1)) \binom {n}{n/2}$, which is asymptotically optimal. For sufficiently large $n$, the extremal families were also determined. In \cite{Long}, the second author proved that this bound almost still applies if we only forbid `tilted pairs' $A,B$ with a single pattern. It was shown that if ${\cal A} \subset {\cal P}[n]$ does not contain $A,B \in {\cal P}[n]$ with $|B \setminus A| \neq 2|A\setminus B|$ for all $A,B \in \A $ and satisfying $a < b$ for all $a\in A\setminus B$ and $b\in B \setminus A$ then $|{\cal A}| \leq C^{\sqrt {\log n}} \binom {n}{n/2}$, for some constant $C> 0$. This condition is equivalent to ${\cal A}$ being $P(d)$-free for all patterns $P(d)$ consisting of $d$ $+$ signs followed by $2d$ $-$ signs. This bound was recently improved by Gerbner and Vizer in \cite{Gerbner-Vizer}. They proved that such families satisfy $|{\cal A}| \leq C \sqrt{\log n } \binom {n}{n/2}$. No family is known for this problem with order more than $C \binom {n}{n/2}$.

Lastly, we mention a fascinating question raised by Johnson and Talbot \cite{Johnson-Talbot} related to Theorem \ref{thm: general pattern theorem} 
(similar conjectures have been raised by Bollob\'as, Leader and Malvenuto \cite{Bollobas-Leader-Malvenuto}, and Bukh \cite{Bukh}). Our phrasing slightly differs from that in \cite{Johnson-Talbot}.
\begin{question*}[Johnson--Talbot]
Is it true that for any $k \in {\mathbb N}$ and $\alpha > 0$ there is $n_0(k,\alpha )\in {\mathbb N}$ with the following property. Suppose that $n \geq n_0(k, \alpha)$ and that ${\cal A} \subset \binom {[n]}{n/2}$ with $|{\cal A}| \geq \alpha \binom {n}{n/2}$. Then there are disjoint sets $S \in \binom {[n]}{n/2 -\lfloor k/2 \rfloor }$ and $T \in \binom {[n]\setminus S}{k}$ such that the family ${\cal C}_{T,S} := \big \{S \cup U: U \in \binom {T}{\lfloor k/2 \rfloor } \big \}$ 
is contained in ${\cal A}$.
\end{question*}
This is true for $k=3$, but is already open for $k=4$. In this case it is possible to guarantee that $|{\cal C}_{T,S} \cap {\cal A}| \geq 5$ for some $T,S$ (note $|{\cal C}_{T,S}| = 6$ for $k=4$). More generally, Johnson and Talbot \cite{Johnson-Talbot} proved that under the hypothesis above, $|{\cal C}_{T,S} \cap {\cal A}| \geq 4\cdot 3^{{(k-4)}/{3}}$ for some $T,S$. We note the conclusion that dense subsets of ${\cal P}[n]$ contain all small patterns ({from Theorem \ref{thm: general pattern theorem}}) would immediately follow from a positive answer to this question. Indeed, for $k= 2d$ any set ${\cal C}_{T,S}$ contains every $d$-balanced pattern. Theorem \ref{thm: general pattern theorem} may be seen as giving (weak) evidence for the question: for $k=2d$ and any $d$-balanced pattern $P$, there is $T$ and $S$ and sets $A,B \in {\cal C}_{T,S} \cap {\cal A}$ with $\pat A B = P$.
\vspace{2mm}

\textbf{Notation:} Given a set $X$, we write ${\cal P}(X)$ for the power set of $X$ and $\binom {X}{k} = \{A \subset X: |A| = k\}$. Given integers $m,n \in {\mathbb N}$ with $m\leq n$, we let $[n] = \{1,\ldots ,n\}$ and $[m,n] = \{m,m+1,\ldots ,n\}$. We also write $(n)_m$ for the falling factorial $(n)_m = n(n-1)\cdots (n-m+1)$.

\section{Small balanced patterns}

In this section we prove Theorem \ref{thm: general pattern theorem}. We will find it convenient to prove many of our results restricted of the middle layer. We then simply write $f(k,P)$ for $f(2k,k,P)$, $\delta (k, P)$ for $\delta (2k,k,P)$, etc.. The following simple observation is useful to move results between different layers of the cube.

\begin{prop}
	\label{prop: moving to different levels}
	Let $n,m ,k,l\in {\mathbb N}$ with $m \leq n$ and $l \leq k$  
	and let $P$ be a pattern. Then $\delta (n,k,P) \leq 
	\delta (m,l,P)$.
\end{prop}

\begin{proof}
	Suppose ${\cal A} \subset \binom {[n]}{k}$ is $P$-free 
	with $|{\cal A}| = \delta (n,k,P)\binom {n}{k}$. Select two 
	disjoint sets $T$ and $U$ of order $m$ and $k-l$ uniformly at random. 
	Then let ${\cal A}_{T,U} = \{A \in \binom {T}{l}: A \cup U \in \A \}$.
	As $\A $ is $P$-free, the set ${\cal A}_{T,U}$ must also be $P$-free 
	for all $T,U$, giving $|{\cal A}_{T,U}| \leq \delta (m,l,P) \binom {m}{l}$. 
	However, ${\mathbb E}_{T,U} |{\cal A}_{T,U}| 
	= \delta (n,k,P)\binom {m}{l}$. The result follows.
\end{proof}

Our next two lemmas are the main steps in the proof of Theorem \ref{thm: general pattern theorem}. Combined they will allow a recursive bound for $\delta (k,d)$ based on bounds on $\delta (k',d')$ for $k' < k$ and $d' < d$. 
	
	\begin{lem} 
		\label{lem: general patterns, differing signs}
	Let $d,k \in {\mathbb N}$ and let $P$ be a 
	$d$-balanced pattern with $P_1 \neq P_{2d}$. Then 
	for any $\gamma \in [\frac {16 \log k}{k^{1/2}} ,1]$ we have 	$$\delta (k,P) \leq \max \Big ( \gamma , 
	6\sqrt { \delta \big ( \lceil {\gamma ^2k}/{64} \rceil,d-1 \big)} \Big ).$$
	\end{lem}

	\begin{proof}
	Let $\gamma $ be chosen as above and suppose that 
	${\mathcal A} \subset \binom {[2k]}{k}$ is $P$-free with 
	$|{\mathcal A}| = \alpha \binom {2k}{k}$. If $\alpha 
	\leq \gamma $ then there is nothing to prove, 
	so we will assume that $\alpha 
	> \gamma \geq \frac {16\log k}{k^{1/2}}$. We will first show that there 
	are many pairs $A,B \in {\mathcal A}$ with $|A \triangle B| = 2$. 
	Indeed, given $C \in \binom {[2k]}{k +1}$ let $y_C$ denote the number of 
	$A\in \A$ with $A \subset C$. Then we have
		\begin{equation*}
			\label{equation: y_C sum bound}
			\sum _{C \in \binom {[2k]}{k +1}} y_C 
				= 
			\big | \big \{(A,C) \in {\mathcal A} \times 
			\binom {[2k]}{k +1}: A \subset C \big \} \big |
				=
			|{\mathcal A}| k 
				\geq 
			\alpha  k \binom {2k}{k+1}.
		\end{equation*} 
	As for every pair $A,B \in {\mathcal A}$ with $|A \triangle B| = 2$ 
	there is a unique set $C \in \binom {[2k]}{k +1}$ with $A,B \subset C$, 
	we obtain 
		\begin{align*}
			\big | \big \{(A,B) \in {\mathcal A}^{(2)}: |A \triangle B| =2
			\big \} \big | 
					& = 
			\sum _{C \in {\binom {[2k]}{k +1}}} \binom {y_C}{2}
			 		\geq 
			\binom {2k}{k +1 } \binom {\alpha k}{2} \\ 
					&  \geq   
			\frac {\alpha ^2 k^2}{4} \times \frac{2k(2k-1)}{(k+1)k}\binom {2k-2}{k -1}
					\geq 
			\frac {\alpha ^2 k^2}{2}\binom {2k-2}{k -1}.
			 \numberthis \label{equation: double counting to get set pair count}
		\end{align*}
	The first inequality holds by the convexity of $\binom {x}{2}$ and the 
	second since $\alpha k -1 \geq \alpha k /2$ as $\alpha \geq 2/ k$.
	
	Now given $1\leq i < j \leq 2k$, let $\A _{i,j}:= 
	\{A \in [2k] \setminus \{i,j\}: A \cup \{i\}, A \cup \{j\} \in \A \}$. 
	Note that from \eqref{equation: double counting to get set pair count} 
	we have 
		\begin{equation}
			\label{equation: A_ij control}
			\sum _{i<j} |\A _{i,j}|
				= 
			\big | \big \{(A,B) \in {\mathcal A}^{(2)}: |A \triangle B| =2 \big \}\big | 
			\geq \frac {\alpha ^2 k^2}{2}\binom {2k-2}{k -1}. 
		\end{equation}
	Also let $\alpha _{i,j}$ and ${\beta }_{i,j}$ be defined so that 
	$|\A _{i,j}| = \alpha _{i,j} 
	\binom {2k-2}{k-1}$ and $\beta _{i,j} = (j-i)/2k$. 
	By \eqref{equation: A_ij control} 
	we find $\{i,j\}$ with $\alpha _{i,j} 
	\geq \frac{\alpha ^2 }{8}$ and $\beta _{i,j} \geq \frac{\alpha ^2}{16}$. 
	Indeed, we have
		\begin{align*}
			\sum _{\substack{\{i,j\} : \alpha _{i,j}< \frac{\alpha ^2 }{8}}} 
				|\A _{i,j}|
				+
			\sum _{\substack{\{i,j\}: \beta _{i,j} < \frac{\alpha ^2}{16}}} 
				|\A _{i,j}| 
				& < 
			\binom {2k}{2} \frac{\alpha ^2}{8} \binom {2k-2}{k -1}
				+
			2k \times \frac {\alpha ^2}{16} 2k
			\binom {2k-2}{k -1}
				\leq 
			\frac {\alpha ^2 k^2}{2} \binom {2k-2}{k -1}.
		\end{align*}
	Combined with \eqref{equation: A_ij control} we see that a claimed pair 
	$\{i,j\}$ exists. Fix such a pair $\{i,j\}$ and set 
	${\cal B} = {\cal A }_{i,j}$.
	
	Now let $X = [i+1,j-1]$ and $Y = [n] \setminus 
	[i,j]$ so that ${\cal B} \subset \binom {X \cup Y}{k-1}$. Partition elements 
	from $\binom {X \cup Y}{k-1}$ according to how they intersect $X$,  
	for each $\ell \in [0,j-i-2]$ letting
		\begin{equation*}
			X_{\ell } 
					= 
			\Big \{ A \in \binom {X \cup Y}{k-1}: 
			\big |A \cap X \big | = \ell \Big \}.
		\end{equation*} 
	Also let ${\cal B}_{\ell } = 
	{\cal B} \cap X_{\ell }$ and
	$L = \Big \{\ell : \big |\ell - \frac {|X|}{2} \big | 
	\leq \sqrt {|X|\log \big (\frac{8}{\alpha }\big )} \Big \}$. 
	By Chernoff's inequality we have 
	$$\sum _{\ell \notin L} |X_{\ell }| \leq 
	\frac {\alpha ^2}{32} \binom {|X| + |Y|}{k-1}.$$
	Using that $|{\cal B}| = \alpha _{i,j} \binom {|X| +|Y|}{k-1} \geq 
	\frac {\alpha ^2}{16} \binom {|X| +|Y|}{k-1}$ this shows that   
		\begin{align*}
			\sum _{\ell \in L} |{\cal B}_{\ell } \big |
				& \geq 
			|{\cal B}| - \frac {\alpha ^2}{32}\binom {|X| + |Y|}{k-1}  \geq 
			\frac {\alpha ^2}{32} \binom {|X| + |Y|}{k -1}
				 \geq 
			\frac {\alpha ^2}{32}
			\sum _{\ell \in L} |X_{\ell }|.
		\end{align*}
	The last inequality here holds since the sets $X_{\ell }$ are 
	disjoint subsets of $\binom {X \cup Y}{k - 1}$. Thus for some $\ell \in L$ 
	we have $|{\cal B}_{\ell }| \geq \frac {\alpha ^2}{32} |X_{\ell }|$. By 
	averaging, we find a set $U \subset Y$ with 
	$|U| = k-\ell-1$ such that the family 
			${\cal C} = \big \{ C \in \binom {X}{\ell }: C \cup U 
			\in {\cal B}_{\ell } \big \}$
	satisfies $|{\cal C}| \geq \frac {\alpha ^2}{32} \binom {|X|}{\ell }$.
	
	To complete the proof, let $Q$ denote the pattern obtained from $P$ by 
	removing $P_1$ and $P_{2d}$, i.e. $Q = P_2\cdots P_{2d-1}$. 
	Note that as $P_1 \neq P_{2d}$ we see that $Q$ is $(d-1)$-balanced. We claim 
	that $\mathcal C$ is $Q$-free. Indeed, suppose $C_1,C_2 \in {\cal C}$ with 
	$\pat {C_1}{C_2} = Q$. Then 
	by definition of ${\cal C}$ and ${\cal B} = {\cal A}_{i,j}$ we have 
	$$\Big \{C_a \cup U \cup \{h\} : a\in \{1,2\}, h\in \{i,j\} \Big \} 
			\subset 
	{\mathcal A}.$$ 
	If $P_1 = +$ we find $\pat {C_1 \cup U\cup \{i\}} {C_2 \cup U\cup \{j\}} 
	= P$.
	If $P_1 = -$ we find $\pat {C_1 \cup U\cup \{j\}} {C_2 \cup U\cup \{i\}} = P$. 
	Thus ${\cal C}$ must be $Q$-free and  
	$$\frac {\alpha ^2}{32} \binom {|X|}{\ell } \leq |{\cal C}| 
	\leq \delta (|X|,\ell ,Q) \binom {|X|}{\ell }.$$ 
	Take $k' = \big \lfloor \frac{|X|}{2} - 
	\sqrt {{|X|} \log \big ( \frac {8}{\alpha } \big ) } 
	\big \rfloor $. A calculation 
	shows that $\frac {|X|}{4} \geq \sqrt {|X| \log \big ( \frac {8}{\alpha } 
	\big ) } + 2$ since $\alpha \geq \frac {16\log k}{k^{1/2}}$ and 
	$|X| +2= \beta _{i,j}2k \geq \frac{\alpha ^2k}{8}$. This gives
	$$ k' \geq \frac{|X|}{2} - 
	\sqrt {{|X|} \log \big ( \frac {8}{\alpha } \big ) } 
	-1
		\geq 
	\frac {|X|}{4} + 1 \geq \Big \lceil \frac {\beta _{i,j}k}{2} \Big \rceil
	\geq \Big \lceil \frac {\alpha ^{2}k}{64} \Big \rceil 
	\geq \Big \lceil \frac {\gamma ^{2}k}{64} \Big \rceil.$$
	
	Since $\ell \in L$ we have $k' \leq \ell \leq |X| - k'$. 
	Using Proposition \ref{prop: moving to different levels} 
	we find that $\frac {\alpha ^2}{32}\leq \delta (|X|,\ell ,d-1) \leq 
	\delta (2k',k',d-1) = \delta (k' ,d-1) \leq 
	\delta (\lceil \frac {\gamma ^{2}k}{64} \rceil ,d-1)$. 
	Rearranging this gives $\alpha \leq 6 \sqrt{ \delta (\lceil 
	\frac {\gamma ^2 k}{64} \rceil , d-1 )}$.
	\end{proof} 
	
	Our second lemma deals with the case where $P$ starts and ends with the same signs.
	
	\begin{lem} 
		\label{lem: general patterns, same signs}
	Let $d \in {\mathbb N}$ and let $P$ be a 
	$d$-balanced pattern with $P_1 = P_{2d}$. Then there are 
	$d_1,d_2 \geq 1$ with $d_1 + d_2 = d$ such that the following 
	holds. For every $k_1, k_2$ with $2k_1 + k_2 = k$ we have
	\begin{equation*}
				\delta (k, P) 
						\leq 
				\max 
					\Big ( 2e^{-k_1/12},
						4 \delta (k_1 ,d_1), 
						4(3k_1)^{2d_1} \delta (k_2,d_2)
				 	\Big ).
		\end{equation*}
	Similarly for every $k_1, k_2$ with $k_1 + 2k_2 = k$ we have
		\begin{equation*}
				\delta (k, P) 
						\leq 
				\max 
					\Big ( 2e^{-k_2/12},
						4 \delta (k_2 ,d_2), 
						4(3k_2)^{2d_2} \delta (k_1,d_1)
				 	\Big ).
		\end{equation*}
	\end{lem}
	
	\begin{proof}	
	To begin, for each $\ell \in [0,2d]$ let 
	$$c_{\ell } 
			= 
	\big| \big\{j \in [\ell ]: P_j = + \big \} \big| - 
	\big| \big \{j \in [\ell ]: P_j = - \big \} \big |.$$ 
	As $P$ is $d$-balanced and $P_1 = P_{2d}$, we have $c_{2d-1} = -c_1$. 
	Combined with the fact that $c_0 = c_{2d}=0$ and $c_{\ell }$ changes by exactly 
	$1$ as ${\ell }$ increases, we see that $c_{2d_1} = 0$ for some 
	$1 \leq d_1 \leq d-1$. Setting $d_2:= d-d_1$ and $Q_1 = P_1\cdots P_{2d_1}$, 
	$Q_2 = P_{2d_1 +1}\cdots P_{2d}$ it is easy to see that these patterns 
	are $d_1$-balanced and $d_2$-balanced respectively.
	
	Now suppose that ${\cal A} \subset \binom {[2k]}{k}$ with $|\A | = 
	\alpha  \binom {2k}{k}$ and that $\A $ is $P$-free. 
	We will prove the first bound above as the 
	second bound is proved identically. We will assume that $\alpha 
	\geq 2e^{-k_1/12}$ as otherwise there is nothing to show.
	Partition $[2k]$ into two consecutive intervals 
	$I_1 = [3k_1]$ and $I_2 = [3k_1 + 1, 2k]$. For each $\ell \in I_1$ 
	let $Z_{\ell } := \binom {I_1}{\ell } \times \binom {I_2}{k - \ell }$.  
	Let $L = \Big \{\ell \in I_1 : |\ell - 3k_1/2| \leq \sqrt {3k_1 \log \big ( 
	\frac {2}{\alpha } \big )} \Big \}$. 
	Note that as $|\bigcup _{\ell \notin  L} Z_{\ell }| \leq 
	\frac {\alpha }{2} \binom {2k}{k}$ by Chernoff's inequality, we have 
	$|\A \cap Z_{\ell }| \geq \frac {\alpha }{2} |Z_{\ell }|$ for some 
	$\ell \in L$. Fix such a choice of $\ell $ and set $Z := Z_{\ell }$ and 
	$\B = \A \cap Z_{\ell }$ so that $\B \subset Z$ with $|\B| \geq 
	\frac {\alpha }{2} |Z|$. 
	
	We will now prove that $\alpha $ satisfies
		\begin{equation}
			\label{equation: double upper bound on delta} 
				\alpha 
						\leq 
				\max 
					\Big ( 
						4 \delta (|I_1|, \ell ,Q_1), 
						4|I_1|^{2d_1} \delta (|I_2|, k-\ell ,Q_2)
				 	\Big ).
		\end{equation}
	To see this, we may assume that $\alpha \geq 
	4 \delta (|I_1|,\ell ,Q_1)$ as otherwise there is nothing to show. 
	Consider the set ${\cal P}_{Q_1}$ given by 
		\begin{equation*}
			{\cal P}_{Q_1} = \Big \{ 
				(A,B) \in Z\times Z : 
				\pat {A \cap I_1} {B \cap I_1} = Q_1  \mbox { and } 
				A \cap I_2 = B \cap I_2 \Big \}.
		\end{equation*}
	We will first show that 
	$|(\B \times \B ) \cap {\cal P}_{Q_1}| \geq \frac{\alpha } {4 |I_1|^{2d_1}} 
	|{\cal P}_{Q_1}|$. Indeed, for each 
		$D \in \binom {I_2}{k- \ell}$ let 
	$${\cal E}({D}) := \big \{C \in \binom {I_1}{\ell}: 
	C \cup D \in {\cal B} \big \}; 			\qquad 
	  {\cal P}_{Q_1}({D}) := 
	  \big \{C,C' \in {\cal E}({D}): \pat {C} {C'} = Q_1 \big \}.$$
	Noting that each ${\cal C} \subset \binom {I_1}{\ell }$ with 
	$|{\cal C}| > \delta (|I_1|,\ell ,Q_1) \binom {|I_1|}{\ell }$ contains 
	$C, C' $ with 
	$\pat {C} {C'} = Q_1$, we find 
	$|{\cal P}_{Q_1}({D})| \geq |{\cal E}(D)| - \delta (|I_1|,\ell ,Q_1) 
	\binom {|I_1|}{\ell }$. Combined these give 
		\begin{equation}
			\label{equation: lower bound on the Q_1 pairs in on half}
			\big |(\B \times \B ) \cap {\cal P}_{Q_1}|  
				= 
		  \sum _{D \in \binom {I_1}{\ell}} | {\cal P}_{Q_1}({D})|
				\geq 
		  \sum _{D \in \binom {I_2}{k- \ell}} 
		  	\Big (|{\cal E}(D)| - 
		  	\delta (|I_1|,\ell ,Q_1) \binom {|I_1|}{\ell } \Big )
		  		\geq 
		  \frac {\alpha }{4} |Z|,
		\end{equation}
	The final inequality here holds since $\sum _{D \in \binom {I_2}{k- \ell}} 
	|{\cal E}(D)| 	= |\B | \geq \frac {\alpha }{2} |Z|$ and 
	$\alpha  \geq 4\delta (|I_1|,\ell ,Q_1)$. Lastly, using that $|{\cal P}_{Q_1}| 
	\leq |I_1|^{2d_1} |Z|$ together with  
	\eqref{equation: lower bound on the Q_1 pairs in on half}, we obtain 
	$|({\cal B} \times {\cal B} )\cap {\cal P}_{Q_1}| \geq \frac{\alpha }
	{4 |I_1|^{2d_1}}|{\cal P}_{Q_1}|$. 
	 
	Now, from this bound we find a choice of $C,C' \in \binom {I_1}{\ell }$ with 
	$\pat C {C'} = Q_1$ such that the set
	$${\cal F}_{C,C'} 
			= 
	\Big \{ D \in \binom {I_2}{k- \ell }: C\cup D, 
	{C'}\cup D \in {\cal B} \Big \}$$ 
	satisfies $|{\cal F}_{C,C'}| \geq 
	\frac{\alpha }
	{4 |I_1|^{2d_1}}\binom {n_2}{k - \ell }$.
	However, if $D,D' \in {\cal F}$ with $\pat D {D'} = Q_2$ 
	then $C\cup D, C '\cup D' \in \A $ and $\pat {C \cup D}{C' \cup D'} = 
	 Q_1Q_2 = P$. As $\A $ is $P$-free we see  
	${\cal F}_{C,C'} \subset \binom {I_2}{k- \ell}$ is $Q_2$-free. This gives 
	$\frac {\alpha }{4|I_1|^{d_1}} \leq \delta (|I_2|,k- \ell ,Q_2)$ 
	and proves \eqref{equation: double upper bound on delta}.
	
	To complete the proof, note that as $\alpha \geq 2 e^{-k_1/12}$, by 
	definition of $L$ we have 
	$\ell \in L \subset [k_1,2k_1]$ and $k - \ell \in [k_2, k_2 + k_1]$. 
	As $|I_1| = 3k_1$ and $|I_2| = 2k - 3k_1 = 2k_2 + k_1$, by Proposition 
	\ref{prop: moving to different levels} we find 
	$$ \delta (|I_1|,\ell ,Q_1) \leq \delta (2k_1,k_1,Q_1) = \delta (k_1,d_1)
	\qquad \mbox{ and } \qquad 
	\delta (|I_2|,k -\ell ,Q_1) \leq \delta (2k_2,k_2,Q_2) = 
	\delta (k_2,d_2).$$
	Combined with \eqref{equation: double upper bound on delta} 
	this completes the proof.
\end{proof}

\begin{proof}[Proof of Theorem \ref{thm: general pattern theorem}]
	We prove by induction on $d$ that with 
	$a_d = (8d)^{5d}$ and $c_d = 6d8^{-d}$ we have
		\begin{equation}
			\label{equation: recursion equation for general pattern}			
			\delta (k,d) 
				\leq
			a_dk^{-c_d}.
		\end{equation}
	For $d =1$ we have 
	$P = +- $ or $P = -+$ and ${\mathcal A} \subset \binom {[2k]}{k}$ 
	is $P$-free simply means that 
	$|A \triangle B| \neq 2$ for all $A,B \in {\cal A}$. 
	It is well known that such families satisfy 
	$|{\cal A}| \leq \frac{1}{k} \binom {2k}{k}$. Indeed, for each 
	$C \in \binom {[2k]}{k +1}$ let $y_C$ denote the number of $A\in \A$ with 
	$A \subset C$. Then 
		\begin{equation*}
			\label{equation: y_C sum bound}
			\sum _{C \in \binom {[2k]}{k +1}} y_C 
				=
			\big | \big \{(A,C) \in {\mathcal A} \times 
			\binom {[2k]}{k +1}: A \subset C \big \} \big | = 
			|{\mathcal A}| \times k.
		\end{equation*}
	However, if $|A \triangle B| \neq 2$ for all $A,B \in {\cal A}$ we must have 
	$y_C \leq 1$ for all $C$. Rearranging, we obtain the claimed upper bound on 
	$|{\cal A}|$. This easily gives that 
	\eqref{equation: recursion equation for general pattern} holds for $d=1$.
	
	We now prove the result for a $d$-balanced pattern $P$, assuming by 
	induction that the theorem holds for all $d'$-balanaced patterns with 
	$d' < d$. We can assume that $k \geq a_d^{1/c_d}$ as otherwise the 
	statement is trivial. We will first prove this when $P$ begins 
	and ends with different signs, using Lemma 
	\ref{lem: general patterns, differing signs}. To apply this, let 
	$\gamma = {8 {(a_{d-1})}^{1/2}}{k^{-\frac{c_{d-1}}{4}}}$
	and note that $\gamma \geq {8 {(a_{d-1})}^{1/2}}{k^{-\frac{1}{4}}} \geq 
	16(\log k ) k^{-1/2}$ since 
	$k^{1/4} /\log k \geq 1/32 \geq 2(a_{d-1})^{-1/2}$. Therefore we can apply 
	Lemma \ref{lem: general patterns, differing signs} 
	to find 
	\begin{align*}
			\delta (k,P) 
					\leq 
			\max \Big (\gamma , 6\sqrt { \delta \big (  \big \lceil \frac{\gamma ^2 k}{64}  \big \rceil , d-1 \big )}\Big ) 
					& \leq 
			\max \Big ({8 {(a_{d-1})}^{1/2}}{k^{-\frac{c_{d-1}}{4}}} , 
			6{\sqrt { {a_{d-1} }{ {\big (a_{d-1}{k^{1-\frac{c_{d-1}}{2}}} \big )^{-c_{d-1}}}}} }
			\Big )\\
					& \leq 
			8 (a_{d-1})^{1/2}k^{-\frac {c_{d-1}}{4}} \leq 
			a_dk^{-c_d}.
	\end{align*}
	The second inequality here uses that Lemma 
	\ref{lem: general patterns, differing signs} holds for $d-1$ by 
	induction, the third that $(a_{d-1})^{-c_{d-1}} \leq 1$ 
	and $1-\frac{c_{d-1}}{2} \geq \frac{1}{2}$ and the last inequality 
	uses that $c_{d} \leq \frac {c_{d-1}}{4}$.
	
	We now move to the case where $P$ starts and ends with 
	the same signs. Given $P$ let $d_1$ and $d_2$ be as in Lemma 
	\ref{lem: general patterns, same signs} so that $d_1 + d_2 = d$ with 
	$d_i \geq 1$. We will assume that $d_1 \leq d_2$ as the other case follows 
	similarly. Let us set $k_1 = \lceil k^{\beta } \rceil $ where 
	$\beta = \frac {c_{d_2} }{ 2d_1 + c_{d_1}}$. Set $k_2 = k - 2k_1 
	\geq k - 4k^{\beta } \geq k - 4k^{1/2} \geq \frac {k}{2}$ for 
	$k \geq 2^6$. Then by Lemma \ref{lem: general patterns, same signs} we 
	have
	\begin{align*}
			\delta (k, P) 
						&\leq 
				\max 
					\Big ( 2e^{-k_1/12},
						4 \delta (k_1 ,d_1), 
						4(3k_1)^{2d_1} \delta (k_2,d_2)
				 	\Big )\\
				 		&\leq 
				\max 
					\Big ( 2e^{-k^{\beta }/12},
						4 a_{d_1} r ^{-\beta c_{d_1}}, 
						4(6k^{\beta })^{2d_1} a_{d_2} \Big ( \frac{k}{2}\Big ) ^{-c_{d_2}}
				 	\Big )\\
				 		&\leq 
				 \max 
					\Big ( 2e^{-k^{\beta }/12},
						4 a_{d_1} k ^{-\beta c_{d_1}}, 
						8^{2d_1 + 3} a_{d_2} k^{2d_1 \beta - c_{d_2}}  \Big )\\
						&\leq 
				 \max 
					\Big ( 2e^{-k^{\beta }/12},
						8^{2d_1+3} a_{d_2} k^{- \frac {c_{d_1}c_{d_2}}{2d_1 + c_{d_1}} } \Big ) \leq a_d k^{-c_{d}}.
	\end{align*}
The first part of the final inequality here uses $a_d \geq 2k^{c_d}$ for 
$k \leq (a_d / 2)^{1/c_{d}}$ and that $e^{-k^{\beta }/12} \leq k^{-c_{d}}$ for $k \geq (a_d / 2)^{1/c_{d}}$. The second part uses that 
$8^{2d_1 +3 }a_{d_2} \leq a_d$ and that since $d = d_1 + d_2$ and $d \leq 2d_2$ we have 
$c_{d} \leq 12 d_2 8^{-d} \leq 
\frac {36 d_1 d_2 8^{-(d_1 + d_2)}}{2d_1 +1} \leq \frac {c_{d_1}c_{d_2}}{2d_1 + 1} 
\leq \frac {c_{d_1}c_{d_2}}{2d_1 + c_{d_1}} $. This completes this case and the proof of the theorem. 
	\end{proof}

\section{Interval patterns}
In this section, we first prove Theorem \ref{thm: interval pattern theorem}. We then give several lower bounds for the case $n= 2k$ depending on value of $d$.

\subsection{Upper Bound on $\delta (n,n/2,\inte d)$}
\begin{proof}[Proof of Theorem \ref{thm: interval pattern theorem}]
	Let $m= \big \lfloor\frac{n}{8d^2} \big \rfloor$.
	We partition $[n]$ into $m$ intervals, $[n] = 
	I_1 \cup \cdots \cup I_m$ with $|I_i| = \lfloor 8d^2 \rfloor$ 
	or $|I_i| = \lceil 8d^2 \rceil$ for all $i \in [m]$.

	Consider 
	the following way of choosing elements from 
	$\binom {[n]}{n/2}$. First select a set $T \subset \binom {[n]} {n/2-d}$ uniformly at random.
	Let $J= \big \{i \in [m]:|I_i\setminus T| \geq d \big \}$.
	As $d <|I_i|/2$, for every $i \in [m]$ we have
	$${\mathbb P}(i \in J)={\mathbb P}(|I_i \setminus T| \geq d)>{\mathbb P}(|I_i \cap T| \leq |I_i|/2)\geq \frac{1}{2}.$$
	If $i \in J$, further select a set $S_i \subset 
	\binom{I_i \setminus T} {d}$ uniformly at random, and 
	set $A_i=T\cup S_i$. If $i \in [m] \setminus J$ simply 
	set $A_i=\emptyset$. 
	
Now for every	$i,j\in J$ with $i<j$, we have $\pat {A_i}{A_j} = \inte d$. Also for $i \notin J$ we have $A_i \notin {\mathcal A}$, since $|A_i| = 0 \neq n/2$. We conclude that there is at most one index $i\in [m]$ with $A_i \in {\mathcal A}$. Equivalently,	
$$\sum_{i=1}^m \mathbf{1}_{A_i \in {\mathcal A}}=\sum_{i=1}^m \sum _{A \in {\mathcal A}} \mathbf {1}_{A_i=A} \leq 1.$$ This is true for any choice of $T$ and $S_i$'s, so in particular if we	take the expectation on both sides, we have
	\begin{equation}
		\label{expectation count}
		\sum _{i=1}^m  \sum _{A \in {\mathcal A}}  
		{\mathbb P}(A_i=A)
		 \leq 1.
	\end{equation}
But as $A_i \notin {\mathcal A}$ for $i \notin J$, given any $A \in {\mathcal A}$ we get that ${\mathbb P}(A_i=A)={\mathbb P}(A_i=A|i \in J){\mathbb P}(i \in J )>\frac{1}{2}{\mathbb P}(A_i=A|i\in J)$. Rewriting (\ref{expectation count}), this gives
\begin{equation}
\label{conditional expectation}
		\sum _{i=1}^m  \sum _{A \in {\mathcal A}}  
		\frac{ {\mathbb P}(A_i=A|i \in J) }{2}
		 \leq 1.
		\end{equation}

\begin{lemma} 
\label{compare}
Let $A \in \binom{[n]}{n/2}$ be a fixed set.
If $|A \cap I_i|\geq \frac{|I_i|}{2} + d$,
then 
 ${\mathbb P}(A_i=A|i \in J) \geq \frac{1}{\binom {n} {n/2}}$.
\end{lemma}

\begin{proof}
Indeed, ${\mathbb P}(A_i=A|i \in J) = \frac {N_i(A)}{N_i}$ where 
	\begin{align*}
		N_i(A) &:= \Big | \Big \{(S_i,T): S_i \in \binom{I_i}{d}, \: T \in \binom {[n]\setminus S_i}{n/2-d},\: S_i \cup T=A \Big \} \Big |;\\
		N_i 	&:= \Big | \Big \{ (S_i,T): S_i \in \binom{I_i}{d}, \: T \in \binom {[n]\setminus S_i}{n/2-d} \Big \} \Big |.
	\end{align*}
However, we have
 \begin{align*}
	\frac {N_i(A)}{N_i} 
		\geq  	
 	\frac {\binom{4d^{2}+d} {d}}{\binom{8d^2} {d}\binom{n-d}  {\frac{n}{2}-d}}
 		=
 	\frac{(4d^{2}+d)_{d}(\frac{n}{2}-d)!\frac{n}{2}!}{(8d^{2})_{d}(n-d)!} 
 		\geq 
 	\frac{(\frac{n}{2}-d)!\frac{n}{2}!}{2^d(n-d)!}
 	>\frac{(n/2)!(n/2)!}{(n)!} = \frac {1}{\binom {n}{n/2}}.
\end{align*}
\end{proof}

For a set  $A \in \binom {[n]} {n/2}$, denote by $G(A)=\big |\big \{i \in [m]: |A \cap I_i|\geq \frac{|I_i|}{2} +d \big \} \big |$. From Lemma \ref{compare} it follows that for any given $A$, we have $\sum _{i =1 }^m {\mathbb P}\big (A_i=A|i \in J \big )\geq G(A) \times \frac{1}{\binom {n} {n/2}}$. Together with (\ref{conditional expectation}), we obtain
\begin{equation} 
\label{probability sum}
\sum _{A \in {\mathcal A}}G(A)
\leq 2\binom {n} {n/2}.
\end{equation}
We call a set $A \in \binom {[n]} {n/2}$ \emph{bad}, if $G(A)<m/5$. Otherwise, we say that $A$ is \emph{good}. Let ${\mathcal B}$ be the family of all bad sets.

\begin{lemma}
$|{\mathcal B}|=o(\frac{1}{m} \binom {n} {n/2})$ for sufficiently large $n$.
\end{lemma}

\begin{proof}
For a uniform random choice of a set $A \subseteq \binom{[n]}{n/2}$, let $X_i$ be a random variable, with $X_i=1$ if $|A\cap I_i|>\frac{|I_i|}{2}+d$, and $X_i =0$ otherwise. Let $Z=X_1+\dots +X_m$. To prove the lemma, we need to show that ${\mathbb P}(Z<m/5)\ll \frac{1}{m}$. By linearity of expectation, ${\mathbb E}Z=m{\mathbb EX_i}=m{\mathbb P}(X_i=1)$. Notice that for every $i\neq j$, $X_i$ and $X_j$ are negatively correlated, since if $A$ has many elements in one interval, it is less likely to have many elements on another interval.
	\begin{align*}
		{\mathbb P}(X_i=0)
			=
		\frac{\sum_{i=0}^{4d^2+d}\binom {8d^2} {i}\binom {n-8d^2}{n/2-i}}{\binom {n}{n/2}} 
			\leq 
		\frac{1}{2}+\frac{\sum_{i=4d^2}^{4d^2+d}\binom {8d^2} {i}
		\binom {n-8d^2}{n/2-i}}{\binom{n}{n/2}}
			\leq 
		\frac{1}{2} + 
		\frac{d\binom{8d^2}{4d^2}\binom{n-8d^2}{n/2-4d^2}}{\binom{n}{n/2}}
			<
		0.79.
	\end{align*}
The second inequality uses Stirling's formula. Therefore ${\mathbb P}(X_i=1)={\mathbb E}X_i>0.21$. Using linearity of expectation gives ${\mathbb E}Z=\sum_{i=1}^m {\mathbb E}X_i>0.21m$.

By a version of the Chernoff-Hoefding bound for negatively correlated variables \cite{Panconesi-Srinivasan}, we deduce that \newline ${\mathbb P}(A \in {\mathcal B})={\mathbb P}(Z<0.2m)<{\mathbb P}(Z-{\mathbb E}Z>0.01m)=o(\frac{1}{m})$, finishing the proof.
\end{proof}

Therefore, if $|{\mathcal A}|\geq \frac{2}{m}\binom {n} {n/2}$, then $|{\mathcal A} \setminus {\mathcal B}| = (1-o(1))|{\mathcal A}|$. Using (\ref{probability sum}), we see that
\begin{equation}
(1-o(1))\frac {m|\mathcal{A}|}{10} \leq \sum _{A \in {\mathcal A}\setminus {\cal B}}G(A)\leq  \sum _{A \in {\mathcal A}}G(A) \leq \binom {n} {n/2}.
\end{equation}
Equivalently $|{\mathcal A}|= O(\frac {1}{m} \binom {n}{n/2}) = O\big (\frac{d^2}{n}  \binom {n} {n/2} \big )$, as required.
\end{proof}

\subsection{Lower Bound on $\delta (n,n/2,\inte d)$}
For the lower bounds, we provide different lower bounds, depending on the range of $d$.
\begin{thm} The following hold:
\begin{enumerate}[(i)]
	\item If $d=o(\sqrt{n})$, there is an $\inte d$-free family ${\mathcal A} \subseteq \binom{[n]}{n/2}$ with $|\A | = \Omega(\max\{\frac{1}{nd},\frac{d^2}{n^{3/2}}\} \cdot \binom{n}{n/2})$.
	\item If $d=c\sqrt{n}$, there is an $\inte d$-free family ${\mathcal A} \subseteq \binom{[n]}{n/2}$ with $|\A | = \Omega _c(\binom{n}{n/2})$.
\end{enumerate}
\end{thm}
\begin{proof}
First we prove \emph{(i)}. For a set $A \in \binom {[n]} {n/2}$ let $S(A):=\sum_{i \in A} i$, the sum of the elements in $A$. Observe that if $\pat A B = \inte d$ then
$0<|S(A)-S(B)|<nd$. Thus for any $0 \leq i \leq nd-1$, the family 
${\mathcal A}_i:= \big \{A \in \binom {[n]}{n/2}|\: S(A) \equiv i (\mod \: nd) \big \}$ forms an $\inte d$-free family. By the pigeonhole principle, we can find such $i$ so that $|{\mathcal A}_i| \geq \frac{1}{nd}\binom{n}{n/2}$.

To obtain the second bound from \emph{(i)}, note that if we choose a set $A \in \binom {[n]}{n/2}$ uniformly at random,  
\begin{equation} 
\label{exp}
{\mathbb E}[S(A)]=\frac{n(n+1)}{4}.
\end{equation}
To calculate the variance, let \[X_i=\begin{cases} 1 &\mbox{if } i \in A \\
0 & \mbox{if } n \notin A. \end{cases} \]
Then $S(A)=\sum_{i=1}^n iX_i$. Now ${\mathbb E}[X_i]=\frac{1}{2}$ every $i \in [n]$ and ${\mathbb E}[X_iX_j]=\frac{1}{4}(1-\frac{1}{n-1})$.
Using this, we find
\begin{align}
	\label{var}
	{\mathbb V}\mbox{ar}( S(A))  = {\mathbb E}[(\sum_{i=1}^n iX_i)^2]-{\mathbb E}[\sum_{i=1}^n iX_i)]^2 & \leq \sum _{i\in [n]}i^2{\mathbb E}[ X_i] + 
\sum _{i \neq j} ij\big ({\mathbb E}[X_iX_j]-{\mathbb E}[X_i]{\mathbb E}[X_j]\big )\nonumber \\	&\leq \sum _{i\in [n]} \frac{i^2}{2} \leq \frac{n^3}{2}.
\end{align} 
%\begin{equation} \label{var}
%\begin{split}
%Var(S(A))={\mathbb E}[(\sum_{i=1}^n iX_i)^2]-{\mathbb E}[\sum_{i=1}^n iX_i)]^2 %\\
%\sum_{i=1}^n i^2{\mathbb E}X_i^2+2\sum_{i\neq j}ij{\mathbb E}[X_iX_j]-
%\frac{n^2(n+1)^2}{16}= \\
%\frac{1}{2}\sum_{i=1}^n i^2 +\frac{(1-\frac{1}{n-1})}{2}\sum_{i \neq j} ij -
%\frac{n^2(n+1)^2}{16}=\\
%\frac{1}{4}(\frac{n-2}{n-1}(\sum_{i=1}^n i)^2+\frac{n}{n-1}\sum_{i=1}^n i^2)-
%\frac{n^2(n+1)^2}{16}=\\
%\frac{1}{4}(\frac{n^2(n+1)^2(n-2)}{4(n-1)}+\frac{n^2(n+1)(2n+1)}{6(n-1)})-
%\frac{n^2(n+1)^2}{16}=\\
%\frac{n^2(n+1)}{16}\cdot(\frac{(n+1)(n-2)+\frac{2}{3}(2n+1)-(n+1)(n-1)}{n-1})
%=\\
%\frac{n^2(n+1)}{16}\cdot\frac{n^2-n-2+\frac{4}{3}n+\frac{2}{3}-n^2+1}{n-1}=\\
%\frac{n^2(n+1)}{16}\cdot\frac{n-1}{3(n-1)}
%=\frac{n^2(n+1)}{48}
%\end{split}
%\end{equation}
From \eqref{exp} and \eqref{var} together, by Chebyshev's inequality we get ${\mathbb P}(|S(A) - n(n+1)/4| \leq n^{3/2}) \geq 1/2$. Equivalently, $|\{A \in \binom {[n]}{n/2}: |S(A) - n(n+1)/4| \leq n^{3/2}\}| \geq \frac {1}{2} \binom {n}{n/2}$. By an easy averaging argument, for some value $m \in [\frac{n(n+1)}{4}-\frac{n^{3/2}}{3},\frac{n(n+1)}{4}+\frac{n^{3/2}}{2}]$.
\begin{equation*}
	|\{A \in \binom {[n]}{n/2}: S(A) \in [m -\frac {d^2}{2}, m +\frac {d^2}{2})\}| 
		\geq 
	\frac{1}{2(2n^{3/2}/d^2 +2)}\binom{n}{n/2} =\Omega \Big (\frac{d^2}{n^{3/2}}\binom{n}{n/2} \Big )
\end{equation*}  However, since two sets $A,B \in \binom {n}{n/2}$ with $\pat A B = \inte d$ have $|S(A) - S(B)| > d^2$, this completes the proof of \emph{(i)}.

To prove \emph{(ii)}, let $c>0$ be given and let $d = c\sqrt n$. Note that if $pat(A,B) = IP(d)$ then for some $i\in [n]$ we have $|A \cap [i]| \geq |B \cap [i]| + d$. This shows that ${\cal A} = \big \{A \in \binom {[n]}{n/2} : \big | |A \cap [i]| - i/2 \big | <d/4 \mbox{ for all }i\in [n]\big \}$ is an $IP(d)$-free family. We will now show that $|{\cal A}| = \Omega _c ( \binom {n}{n/2} )$.

To see this, it is convenient to identify elements of $\binom {[n]}{n/2}$ with certain walks. Let ${\cal W}_0$ denote the set of all walks $W = W_0\cdots W_n$ of length $n$ on ${\mathbb Z}$ with $W_0 = W_n = 0$ and which  either increase or decrease by $1$ in each step (i.e. $|W_i - W_{i-1}| = 1$ for all $i\in [n]$). Note that each walk $W \in {\cal W}_0$ naturally corresponds to a subset of $[n]$ of size $n/2$ consisting of those steps in $[n]$ where the walk increases. Under this correspondence, the set ${\cal A}$ corresponds to those walks in ${\cal W}_0$ which lie entirely in $[-d/4,d/4]$. 

Now select a walk $W \in {\cal W}_0$ uniformly at random. Letting $T$ denote a value to be determined, consider the following events: 
	\begin{align*}
		A  &= \big \{W_j \in [-d/4,d/4] \mbox { for all }j\in [n]\big \}\\
		B  &= \big \{W_{in/T} \in [-d/12,d/12] \mbox { for all } i\in [T-1]\big  \}\\
		C_i  &= \big \{W_{j} \in [-d/4,d/4] \mbox { for all } j\in \big [\frac{(i-1)n}{T}, \frac {in}{T} \big ] \big \},  
		\mbox{ where } i\in [T].
	\end{align*}
Also for $i\in [T-1]$ and $a_i \in [-d/12, d/12]$, let $B_i(a_i)$ denote the event
$B_i(a_i) = \big \{W_{in/T} = a_i\big \}$. We will show that 
	\begin{align}
	\label{equation: breakdown probability}
	{\mathbb P}_{W \sim {\cal W}_0}\Big (B \wedge \bigwedge _{i\in [T]} C_i \Big ) 
		\geq c' > 0,
	\end{align}
where $c'$ depends only on $c$. Since $\bigwedge _{i\in [T]} C_i \subset A$, this will prove the result. 

To begin, note that we have 
	\begin{alignat}{2}
		\label{equation: probability breakdown into events}
			{\mathbb P}_{W \sim {\cal W}_0}
			\Big (B \wedge \bigwedge _{i\in [T]} C_i \Big ) 
		&\geq 
			\sum _{a_1, \ldots ,a_{T-1} \in [-d/12,d/12]} 
			&& {\mathbb P}_{W \sim {\cal W}_0}
			\Big (\bigwedge _{i\in [T-1]}B_i(a_i) 
			\wedge \bigwedge _{i\in [T]} C_i \Big ) \nonumber \\
		& =
			\sum _{a_1, \ldots ,a_{T-1} \in [-d/12,d/12]} 
			&& {\mathbb P}_{W \sim {\cal W}_0}
			 \Big (\bigwedge _{i\in [T]} C_i \big | 
			\bigwedge _{i\in [T-1]}B_i(a_i) \Big )\nonumber \\
		& 
			&& \times {\mathbb P}_{W \sim {\cal W}_0} 
			\Big (\bigwedge _{i\in [T-1]}B_i(a_i) \Big ).
	\end{alignat}
Let ${\cal W}(a,b)$ denote the collection of random walks of length $n/T$ which start at $a$ and end at $b$. Since $C_i$ depends only on $\{W_j: j\in [(i-1)n/T, in/T]\}$, taking $a_0 = a_T = 0$ we have 
	\begin{align}
		\label{equation: breakdown to exploit independence}
			{\mathbb P}_{W \sim {\cal W}_0}
			\Big (\bigwedge _{i\in [T]} C_i \big | 
			\bigwedge _{i\in [T-1]}B_i(a_i) \Big ) 
		& =
			\prod _{i\in [T]} 	
			{\mathbb P}_{W \sim {\cal W}_0}
			\Big (C_i \big | B_{i-1}(a_{i-1}) \wedge B_i(a_i) \Big ) \nonumber \\
		& = 
			\prod _{i\in [T]} 	
			{\mathbb P}_{W \sim {\cal W}(a_{i-1},a_i)}
			\Big (W \mbox{ lies entirely in } [-d/4,d/4] \Big ).
	\end{align}	
	
\textbf{Claim:} For every $a, b \in [-d/12,d/12]$ 
we have ${\mathbb P}_{W \sim {\cal W}(a,b)}
	\Big (W \mbox{ lies entirely in } [-d/4,d/4] \Big ) \geq 1/2$.
	
Let ${\cal W}(a)$ denote the collection of all walks of length $n/T$ which begin at $a$. Let us select $W$ from ${\cal W}(a)$ uniformly at random and let $S_{n/T}$ denote the final vertex. By the reflection principle for random walks, we have
\begin{align*}
\mathbb{P}_{W \sim {\cal W}(a,b)}(W \mbox { exceeds } d/4) &= 
\mathbb{P}_{W \sim {\cal W}(a)}(W \mbox { exceeds }d/4 |S_{n/T}=b)\nonumber \\
& =\frac{\mathbb{P}_{W \sim {\cal W}(a)} (S_{n/T}=d/2-b)}{\mathbb{P}_{W \sim {\cal W}(a)} (S_{n/T}=b)}\nonumber \\ 
& = 
	\frac { \binom {{n}/{T}}{{n}/{2T} + (d/2 - b) - a}}{\binom {{n}/{T}}{{n}/2T + b - a}} 
	\leq 
	\frac { \binom {{n}/{T}}{{n}/{2T} + d/3}}{\binom {{n}/{T}}{{n}/2T + d/6}}\nonumber \\
	& = \frac {(n/2T - d/6 )_{d/6}}{(n/2T + d/6 )_{d/6}} \leq 
	\Big (1 - \frac{dT}{3n} \Big )^{d/6} \leq e^{-d^2T/36n}.
\end{align*}
Taking $T = 72/c^2$ say, we find $\mathbb{P}_{W \sim {\cal W}(a,b)}(W \mbox { exceeds } d/4) \leq e^{-2} < 1/4$. By symmetry, this gives 
${\mathbb P}_{W \sim {\cal W}(a,b)} 
\Big (W \mbox{ lies entirely in } [-d/4,d/4] \Big ) \geq 1 - 2 \times (1/4) = 1/2$, as claimed.

Now by combining \eqref{equation: breakdown to exploit independence} together with the claim in \eqref{equation: probability breakdown into events} we find 
	\begin{equation} 
		\label{equation: reduced probability bound}
		{\mathbb P}_{W \sim {\cal W}_0}
			\Big (B \wedge \bigwedge _{i\in [T]} C_i \Big ) 
		\geq
		\sum _{a_1, \ldots ,a_{T-1} \in [-d/12,d/12]} 
			2^{1-T} \times {\mathbb P}_{W \sim {\cal W}_0} 
			\Big (\bigwedge _{i\in [T-1]}B_i(a_i) \Big ).
	\end{equation}
But letting $b_i := \frac{n}{2T} + a_i - a_{i-1}$ for all $i\in [T]$ where $a_0 = a_T = 0$, we have 
\begin{equation*}
{\mathbb P}_{W \sim {\cal W}_0} \Big (\bigwedge _{i\in [T-1]}B_i(a_i) \Big ) 
= \frac { \prod _{i\in [T]} \binom {n/T}{b_i}} {\binom {n}{n/2}} = \Omega _{c,T} (d^{1-T}).
\end{equation*}
The final inequality follows by Stirling's approximation, using that $b_i \in [\frac{n}{2T} - \frac{d}{6}, \frac{n}{2T} + \frac{d}{6}]$ for all $i\in [T]$. Combined with \eqref{equation: reduced probability bound}, this gives 
${\mathbb P}_{W \sim {\cal W}_0} \Big (B \wedge \bigwedge _{i\in [T]} C_i \Big ) 
= \Omega _{c,T}(1) = \Omega _c(1)$, as required.
\end{proof}

\section{Alternating patterns}

To begin, we prove an auxiliary lemma. Given ${\bf x} = (x_i)$ and ${\bf y} = (y_i)$ in $[m]^D$ we say that ${\bf y}$ $d$-dominates ${\bf x}$ if $|\{i\in [D]: x_i \neq y_i\}| = d$ and 
$x_i \leq y_i $ for all $i\in [D]$.

\begin{lem}
	\label{lem: auxillary m-cube lemma}
	Let $d, m, D \in {\mathbb N}$ with $2md^2 \leq D$. 
	Suppose that ${\cal C} \subset [m]^D$ does 
	not contain	${\mathbf x}$ and ${ \mathbf y}$ 
	such that ${\bf y}$ $d$-dominates ${\bf x}$. 
	Then $|{\cal C}| \leq 2m^{D-1}$.
\end{lem}

\begin{proof}
To begin, choose a set $S \subset [D]$ with $|S| = d$ and a vector ${\mathbf z} \in [m]^{[D]\setminus S}$ uniformly at random. For each $i\in [m]$ let ${\bf z}_S(i) \in [m]^D$ denote the vector which agrees with ${\bf z}$ on coordinates in $[D] \setminus S$ and equals $i$ everywhere else. Also let ${\cal B}_{S, {\bf z}}$ denote the combinatorial line ${\cal B}_{S, {\bf z}} := \{{\bf z}_S(i): i\in [m]\}$.

Now as ${\cal C}$ does not contain any $d$-dominating pairs, for any choice of $S$ and ${\bf z}$ we have 
$|{\cal C} \cap {\cal B}_{S, {\bf z}}| \leq 1$. Letting 
$X_i$ denote the indicator random variable which is $1$ if ${\bf z}_S(i) \in {\cal C}$ and $0$ otherwise, this gives 
	\begin{equation*}
		\sum _{i\in [m]} X_i \leq 1.
	\end{equation*}
Taking expectations over all choice of $S$ and ${\bf z}$, this gives
	\begin{equation}
		\label{equation: probability C lying on the line}
		 \sum _{C \in {\cal C}} 
		{\mathbb P}( C \in {\cal B}_{S, {\bf z}} ) 
			= 
		\sum _{i\in [m]} \sum _{C \in {\cal C}} 
		{\mathbb P}( {\bf z}_S(i) = C ) 
			\leq 
		1.
	\end{equation}
However, an easy calculation gives that if $C$ has $k_i$ entries $i$ for all $i\in [m]$, then
\begin{equation*}
		{\mathbb P}( C \in {\cal B}_{S, {\bf z}} )  
			= 
		\sum _{i\in [m]} \frac {\binom {k_i}{d}}{m^{D-d}\binom {D}{d}}.
\end{equation*}
This expression is minimized when all $k_i$ are as equal as possible. Thus 
\begin{align*}
			{\mathbb P}( C \in {\cal B}_{S, {\bf z}} )  
				\geq 
			m \frac {\binom {D/m}{d}}{m^{D-d}\binom {D}{d}} 
			 	= 
			m \frac {(D/m) _ d }{m^{D-d}D_{d}} 
				& = 
			\frac{m}{m^D} \prod _{l \in [0,d-1]} 
			\Big ( 1 - \frac {l(m-1)}{D - l} \Big ) \\
				& \geq 
			\frac{1}{m^{D-1}} 
			\Big ( 1 - \sum _{l \in [0,d-1]}\frac {l(m-1)}{D/2} 
				 \Big )\\
				 & \geq 
			\frac{1}{m^{D-1}} 
			\Big ( 1 - \frac {md^2}{D} \Big ) 
				\geq 
			\frac{1}{2m^{D-1}}.
\end{align*}
The final line here used $2md^2 \leq D$. Combined with 
\eqref{equation: probability C lying on the line} this gives $|{\cal C}|/2 m^{D-1} \leq 1$, as required. \end{proof}

We are now ready for the proof of Theorem \ref{thm: alternating pattern theorem}.

\begin{proof}[Proof of Theorem \ref{thm: alternating pattern theorem}]
By Proposition \ref{prop: moving to different levels} it suffices to prove the theorem for $n = 2k$. Let $m = \lfloor \frac {\log _2(n/d^2)}{2}\rfloor $. For convenience we assume that $n$ is divisible by $m$, with $Km = n$. Let $[n] = \bigcup _{i=1}^K I_i$ be a partition of $[n]$ where $I_i = \{(i-1)m, \ldots , im-1\}$ for all $i\in [K]$. Given a set $T \subset [K]$, let $T^c = [K]\setminus T$ and let 
	\begin{align*}
		{\cal B}_T &:= \{A \subset \bigcup _{i\in T^c} I_i: 
		|A \cap I_i| \neq 1 \mbox { for all } i\in T^c\}.
	\end{align*}
Given $B \in {\cal B}_T$ and ${\bf x} \in [m]^{T}$ we also let $B({\bf x}) := B \cup \{(i-1)m + j -1: i\in T, x_i = j\}$ and 
	\begin{align*}
		{\cal C}_{B} &:= \{B({\bf x}): 
		{\bf x} \in [m]^{T}\}.
	\end{align*}
Note that for every $A \subset [n]$ there is a unique $T \subset [K]$, $B \in {\cal B}_T$ and ${\bf x} \in [m]^{T}$ such that $A = B({\bf x})$. Thus we have the disjoint union
	\begin{equation}
		\label{equation: powerset decomposition}
		\binom {[n]}{n/2}
			=  
		\bigcup _{T \subset [K]} 
		\bigcup _{\substack{B \subset {\cal B}_T\\ 
		|B| = \frac{n}{2} - |T|}} 
		{\cal C}_{B}.
	\end{equation}	
We will first show that almost all sets $A$ in ${\binom {[n]}{n/2}}$ are of the form $A = B({\bf x})$ where 
$T \subset [K]$ and $B \in {\cal B}_T$ with $|T| \geq mK/2^{m+1}= n/ 2^{m+1}$. To see this, given a set $A \subset [n]$, let $A_i = A \cap I_i$ for all $i\in [K]$. We will say that $A \subset [n]$ is \emph{bad} if $T(A) = \{i\in [K]: |A_i| =1\}$ satisfies 
$|T(A)| \leq \frac {m}{2^{m+1}}K$. We claim that there are at most $O({e^{-n^{1/2}/2}}{2^n})$ sets are bad. Indeed, if we select $A \subset [n]$ uniformly at random, we have ${\mathbb P}(|A_i| = 1) = m/2^m$, which gives ${\mathbb E}(|T(A)|) = \frac {mK}{2^m} = \frac{n}{2^m}$. 
As $|A_i| = 1$ for each $i\in [K]$ independently, by Chernoff's inequality, we find that ${\mathbb P}\big (|T(A)| - \frac {n}{2^m} \leq 
- \frac {n}{2^{m+1}} \big ) \leq e^{- \frac {n}{2^{m+1}}}$. 
As $m \leq \log _2(n/d^2)/2 \leq  \frac {\log _2 n }{2}$ we find that 
${\mathbb P}(A \mbox { is bad} ) \leq e^{-n^{1/2}/2}$. Equivalently, $|\{A \subset [n]: A \mbox { is bad}\}| 
= O(e^{-n^{1/2}}2^n)$.

Now suppose that $T \subset [K]$ with $|T| \geq n/2^{m+1}$ and $B \in {\cal B}_T$. Note that given ${\bf x}, {\bf y} \in [m]^T$, if ${\bf y}$ $d$-dominates ${\bf x}$ then 
$\pat {B({\mathbf x})} {B({\mathbf y})} = \alt d$. Noting that as $m = \lfloor \log _2(n/d^2)/2 \rfloor $ we have $|T| \geq n/2^{m+1} \geq 2^md^2 \geq 2md^2$. Setting $D = |T|$, Lemma \ref{lem: auxillary m-cube lemma} therefore shows that any ${\cal A} \subset \binom {[n]}{n/2}$ which is $\alt d$-free satisfies
\begin{align}
\label{equation: sectional bound} 
|{\cal A} \cap {\cal C}_B| \leq 2m^{|T|-1} = \frac {2}{m} |{\cal C}_B|.
\end{align}
Summing over all $T\subset [K]$ and $B \in {\cal C}_T$, combined with \eqref{equation: powerset decomposition} and 
\eqref{equation: sectional bound}, this gives 
	\begin{align*}
		|{\cal A}|  
			\leq 
		\sum _{T \subset [K]} | {\cal A} \cap 
		\bigcup _{\substack{B \in {\cal B}_T\\ 
		|B| = n/2 - |T|}} {\cal C}_B| 
			& \leq 
		\big | \big \{A\subset [n]: A \mbox{ bad} \big \} 
		\big | 
			+
		\sum _{\substack{T \subset [K]:\\ 
		|T| \geq 2md^2}} 
		\sum _{B \in {\cal B}_T}
		|{\cal A} \cap {\cal C}_B| 	\\
			& \leq 
		O\Big (\frac {2^n}{{e^{{\sqrt{n}}/{2}}}}\Big ) 
			+
		\sum _{\substack{T \subset [K]:\\ 
		|T| \geq 2md^2}} 
		\sum _{\substack{B \in {\cal B}_T\\ 
		|B| = n/2 - |T|}}
		\frac {2}{m}|{\cal C}_B|\\
		& \leq 
		\frac {2 + o(1)}{m} \binom {n}{n/2}.
	\end{align*}
This completes the proof of the theorem.
\end{proof}

\section{Concluding remarks and open problems}

In this paper we proved bounds on the size of families $\A \subset {\cal P}[n]$ which avoid a $d$-balanced pattern $P$. Our proof shows that such families satisfy 
	\begin{equation*}
		|\A | = O(a_dn^{-c_d}2^n),
	\end{equation*}
where $a_k = (8d)^{5d}$ and $c_d = 6d8^{-d}$. In particular, families $\A$ which avoid a $d$-balanced pattern with 
$d < c \log \log n$ satisfy $|{\cal A}|  = o(2^n)$ for some absolute constant $c>0$. It would be interesting to improve the density bound here and/or extend the range of $d$ for which this zero density property holds.

Another interesting question is the following: which balanced pattern $P$ has the strongest effect on the density of $P$-free families ${\cal A} \subset {\cal P}[n]$? That is, what is $\min _{P} \delta (n,k,P)$, where the minimum is taken over all balanced patterns $P$? If instead of patterns we only forbid intersection sizes (as discussed in the Introduction) then there are a number of very strong density results for subsets of ${\cal P}[n]$. For example, the Frankl-R\"odl \cite{Frankl-Rodl} theorem shows that given $\epsilon >0$, if ${\cal A} \subset {\cal P}[n]$ and $|A \cap B| \neq t$ for some $\epsilon n \leq t \leq (1/2 - \epsilon )n$ then 
$|{\cal A}| \leq (2-\delta )^n$, where $\epsilon  = \epsilon (\delta )>0$. It would be very interesting to know if there exists a pattern which forces a superpolynomial density in $n$. That is, does there an increasing sequence of naturals $(n_{k})_{k\in {\mathbb N}}$ and balanced patterns $(P_k)_k$ with $\delta (n_k,n_k/2,P_k) = n_k^{-\omega _k(1)}$ for some function $\omega _k(1)$ tending to infinity with $k$?

Lastly, how large can $d$ be (as a function of $n$) while still giving $\delta (n, n/2, \alt d) \to 0$ as $n \to \infty$. Theorem \ref{thm: alternating pattern theorem} proves that this holds for any $d = o(\sqrt n)$.

\end{document}